\documentclass[12pt]{article}
\usepackage{amssymb,latexsym,amsmath,enumerate,verbatim,amsfonts,amsthm}
\usepackage[active]{srcltx}
\numberwithin{equation}{section}
\newtheorem{Theorem}{Theorem}[section]

\newtheorem{Lemma}{Lemma}[section]
\newtheorem{Corollary}[Theorem]{Corollary}

\makeatletter
 \def\@biblabel#1{#1.}
\makeatother

\newenvironment{remark}[1][Remark]{\begin{trivlist}
\item[\hskip \labelsep {\bfseries #1}]}{\end{trivlist}}

\begin{document}
\title{Infinite dimensional Hilbert tensors on spaces of analytic functions}

\author{Yisheng Song\thanks{School of Mathematics and Information Science  and Henan Engineering Laboratory for Big Data Statistical Analysis and Optimal Control,
 Henan Normal University, XinXiang HeNan,  P.R. China, 453007.
Email: songyisheng@htu.cn. This author's work was supported by
the National Natural Science Foundation of P.R. China (Grant No.
11571095, 11601134) } and Liqun Qi \thanks{Department of Applied Mathematics, The Hong Kong Polytechnic University, Hung Hom, Kowloon, Hong Kong. Email: liqun.qi@polyu.edu.hk. This author's work was  supported by
the National Natural Science Foundation of P.R. China (Grant No.
11571095) and by the Hong Kong Research Grant Council (Grant No. PolyU
501212, 501913, 15302114 and 15300715).}}

\date{\today}

 \maketitle

%---------------------------------------------------------------------------------Abstract
\begin{abstract}
\noindent  %\vspace{3mm}
In this paper, the $m-$order infinite  dimensional Hilbert tensor (hypermatrix) is intrduced to  define an   $(m-1)$-homogeneous operator on the spaces of analytic functions, which is called Hilbert tensor operator. The boundedness of Hilbert tensor operator is presented on Bergman spaces $A^p$ ($p>2(m-1)$). On the base of the boundedness, two  positively homogeneous operators are introduced to the spaces of analytic functions, and hence the upper bounds of norm of such two operators are found on  Bergman spaces $A^p$ ($p>2(m-1)$). In particular,  the norms of such two operators on Bergman spaces $A^{4(m-1)}$ are smaller than or equal to $\pi$ and
$\pi^\frac1{m-1}$, respectively.

\noindent {\bf Key words:}\hspace{2mm} Hilbert tensor,  Analytic function, Upper bound,   Bergman space, Gamma function.  \vspace{3mm}

\noindent {\bf AMS subject classifications (2010):}\hspace{2mm}30H10, 30H20, 30H05, 30C10,
47H15, 47H12, 34B10, 47A52, 47J10, 47H09, 15A48, 47H07.
  \vspace{3mm}

\end{abstract}

%------------------------------------------------------------------------------Section 1
\section{\bf Introduction}\label{}
The Hilbert matrix $H$ is a matrix with entries $H_{ij}$ being the unit fractions  for nonnegative integers $i,j$, i.e.,
$$H_{ij}=\frac1{i+j+1},\ i,j=0,1,2,\cdots$$
which was introduced by Hilbert \cite{H1894}. Let $i,j=0,1,2,\cdots,n$. Then such an $n$-dimensional Hilbert matrix is a compact linear operator on finite dimensional space $\mathbb{R}^n$. The properties of $n$-dimensional Hilbert matrix  had been studied by Frazer \cite{F1946} and Taussky \cite{T1949}. An
infinite dimensional Hilbert matrix $H$ may be regarded as a bounded linear operator from the sequence space
$l^2$ into itself, but not compact operator (Choi \cite{C1983}) and Ingham \cite{I1936}).  Magnus \cite{M1950} and Kato \cite{K1957} showed the spectral properties of such a class of matrice. The infinite dimensional Hilbert matrix $H$ induces an operator defined on the sequence space
$l^p$ ($1\leq p$), for  $x=(x_k)_{k=0}^\infty\in l^p$,
\begin{equation}\label{eq:11}
H(x)=\left(\sum_{j=0}^{+\infty}\frac{x_j}{i+j+1}\right)_{i=0}^\infty.
\end{equation}
If $1<p<+\infty$, the well-known Hilbert inequality (\cite{HLP}) implies that $H$ is an operator on $l^p$ and its operator norm $\|H\|=\sup\limits_{\|x\|_{l_p}=1}\|H(x)\|_{l_p}$ is the following: \begin{equation}\label{eq:12}\|H\|=\frac{\pi}{\sin(\frac{\pi}{p})},\ 1<p<+\infty.\end{equation}
On the other hand, the infinite dimensional Hilbert matrix $H$ also induces a bounded operator  on the spaces of analytic functions defined by
\begin{equation}\label{eq:13}
H(f)(z)=\sum_{i=0}^{+\infty}\left(\sum_{j=0}^{+\infty}\frac{a_j}{i+j+1}\right)z^i,
\end{equation}
for all analytic functions $f(z)=\sum\limits_{k=0}^{+\infty}a_kz^k$ with the convergent coefficients $\sum\limits_{j=0}^{+\infty}\frac{a_j}{i+j+1}$ for each $i$. In Hardy spaces $H^p$, Diamantopoulos and Siskakis \cite{DS2000} proved that $H$ is  bounded for $p>1$ and found an upper boundedness of its operator norm. In 2004,  Diamantopoulos \cite{D2004} showed  that $H$ is  bounded on Bergman spaces for $p>2$ and obtained the upper boundedness of its operator norm. Aleman,  Montes-Rodríguez, Sarafoleanu \cite{AMS} studied the eigenfunctions of Hilbert matrix operator on Hardy space $H^p$ ($p>1$).

As a natural extension of a Hilbert matrix,
the entries of an $m$-order infinite dimensional Hilbert  tensor (hypermatrix) $\mathcal{H}=(\mathcal{H}_{i_1i_2\cdots i_m})$ are defined by
$$\mathcal{H}_{i_1i_2\cdots i_m}=\frac1{i_1+i_2+\cdots+i_m+1},\ i_k=0,1,2,\cdots,\ k=1,2,\cdots,m.$$
Each entry of $\mathcal{H}$ is derived from the integral
\begin{equation}\label{eq:14} \mathcal{H}_{i_1i_2\cdots i_m}=\int_0^1 t^{i_1+i_2+\cdots+i_m}dt. \end{equation}
Clearly, $\mathcal{H}$ are positive ($\mathcal{H}_{i_1i_2\cdots i_m}>0$) and symmetric ($\mathcal{H}_{i_1i_2\cdots i_m}$ are invariant for any permutation of the indices),  and  $\mathcal{H}$ is a Hankel tensor with $v=(1,\frac12, \frac13, \cdots,\frac{1}{n},\cdots)$ (Qi \cite{Q2015}). Song and Qi \cite{SQ2014} studied  infinite and finite dimensional Hilbert tensors, and showed that $\mathcal{H}$ defines a bounded and positively $(m-1)$-homogeneous operator  from $l^1$ into $l^p$ ($1<p<\infty$), and found the upper boundedness of corresponding positively homogeneous operator norm.

A real $m$-order $n$-dimensional tensor (hypermatrix) $\mathcal{A} = (a_{i_1\cdots i_m})$ is a multi-array of real entries $a_{i_1\cdots
	i_m}$, where $i_j \in \{1,2,\cdots,n\}$ for $j \in \{1,2,\cdots,m\}$. Denote the set of all
real $m$th order $n$-dimensional tensors by $T_{m, n}$. Then $T_{m,
	n}$ is a linear space of dimension $n^m$. Let $\mathcal{A} = (a_{i_1\cdots
	i_m}) \in T_{m, n}$. If the entries $a_{i_1\cdots i_m}$ are
invariant under any permutation of their indices, then $\mathcal{A}$ is
called a  symmetric tensor.     Let $\mathcal{A} =(a_{i_1\cdots i_m}) \in
T_{m, n}$ and ${\bf x} \in \mathbb{R}^n$. Then $\mathcal{A} {\bf x}^{m-1}$ is a vector in $\mathbb{R}^n$ with
its $i$th component as
$$\left(\mathcal{A} {\bf x}^{m-1}\right)_i: = \sum_{i_2, \cdots, i_m=1}^n a_{ii_2\cdots
	i_m}x_{i_2}\cdots x_{i_m}$$ for $i \in \{1,2,\cdots,n\}$ (\cite{LQ1}).

For $m-$order finite dimensional tensor, various structured tensors were studied well. For more details, M-tensors see Zhang, Qi and Zhou \cite{ZQZ} and
Ding, Qi and Wei \cite{DQW}; P-(B-)tensors see Song and Qi \cite{SQ-15}, Qi and Song \cite{QS}; copositive tensors see  Song and Qi \cite{SQ 15}; Cauchy tensor see Chen and Qi \cite{CQ};  the applications in nonlinear complementarity problem see Song and Qi \cite{SQ15}, Che, Qi, Wei \cite{CQW}, Song and Yu \cite{SY15}, Luo, Qi and Xiu \cite{LQX}, Gowda, Luo, Qi and Xiu \cite{GLQX}, Bai, Huang and Wang \cite{BHW}, Wang, Huang and Bai \cite{WHB}, Ding, Luo and Qi \cite{DLQ}, Suo and Wang \cite{HSW}, Song and Qi \cite{SQ13}, Ling, He, Qi \cite{LHQ2015, LHQ15}, Chen, Yang, Ye \cite{CYY}.\\

In this paper, we show that an $m-$order infinite  dimensional Hilbert tensor defines an   $(m-1)$-homogeneous operator on the spaces of analytic functions (Hardy spaces $H^p$ ($p>m-1$) and Bergman spaces $A^p$ ($p>2(m-1)$)),
\begin{equation}\label{eq:15}
\mathcal{H}(f)(z)=\sum_{k=0}^{+\infty}\left(\sum_{i_2,i_3,\cdots,i_m=0}^{+\infty}\frac{a_{i_2}a_{i_3}\cdots a_{i_m}}{k+i_2+i_3+\cdots+i_m+1}\right)z^k
\end{equation}
for all analytic functions $f(z)=\sum\limits_{k=0}^{+\infty}a_kz^k.$  The upper bound of Hilbert tensor operator $\mathcal{H}(f)$ is found on Bergman spaces $A^p$ ($p>2(m-1)$) with the help of the proof technique of Diamantopoulos \cite{D2004}. So two  positively homogeneous operators may be defined on Bergman spaces $A^p$
by the formula
\begin{equation} \label{TF}
T_\mathcal{H}(f) (z) := \begin{cases}\|f \|_{A^{p(m-1)}}^{2-m}\mathcal{H}(f)(z),\ f \neq0\\
0,\ \ \ \ \ \ \ \ \ \ \  \ \ \ \ \ \ \ \ \ \ \ \
f =0\end{cases}\mbox{and } F_\mathcal{H}(f) (z) :=\left(\mathcal{H}(f)(z)\right)^{\frac1{m-1}}(m \mbox{ is even}),
\end{equation}
where $T_\mathcal{H} : A^{p(m-1)} \to A^p$ and $F_\mathcal{H} :  A^{p} \to  A^p$. We  obtain the upper bound of operator norm $\|T_\mathcal{H}\|$ and $\|F_\mathcal{H}\|$. In particular, when $p=4(m-1)$,  \begin{equation}\label{eq:17}\|T_\mathcal{H}\|\leq \pi \mbox{ and }\|F_\mathcal{H}\|\leq \pi^\frac1{m-1}.\end{equation}

%%%%%%%%%%%%%%%%%%%%%%%%%%%%%%%%%%%%%%%%%%%%%%%%%%%%%%%%%%%%%%%%%%%%%%%%%%%%%%%%%%%%%%%%%%%%%

The paper is organized as follows:
In Section 2, we will give some basic definitions and facts, which will be used to the proof of main results. In Section 3, we first study the definition of Hilbert tensor operator and give the corresponding proof to show that such an operator is  well-defined. We prove the intergral form of Hilbert tensor operator. Secondly, the boundedness of Hilbert tensor operator is proved  on Bergman spaces $A^p$ ($p>2(m-1)$) by means of its intergral form.  Finally, we define two  positively homogeneous operators induced by $m$-order infinite dimensional Hilbert  tensor and prove the upper boundedness of their operator norm.

%%%%%%%%%%%%%%%%%%%%%%%%%%%%%%%%%%%%%%%%%%%%%%%%%%%%%%%%%%%%%%%%%%%%%%%%%%%%%%%%%%%%%%%%%%%%

\section{Preliminaries and basic facts}

%\hspace{4mm}
In this section, we will  collect some basic definitions and facts, which will be used later on. Throughout this paper,  let $\mathbb{C}$ be the complex plane, and let
$$\mathbb{B}: = \{z\in\mathbb{C}: \|z\| < 1\} $$
be the open unit disk in $\mathbb{C}$. Likewise, we write $\mathbb{R}$ for the real line.  The normalized
Lebesgue measure on $\mathbb{B}$ will be denoted by $d\mu$.  Obviously, $$d\mu(z)=\frac1\pi dxdy=\frac1\pi r drd\theta$$ for $z=x+yi=re^{i\theta}.$  For $0 < p < +\infty$, the Bergman space $A^p$ is a space of all
 analytic functions $f$ in $\mathbb{B}$ with
\begin{equation}\label{eq:21}
\|f\|_{A^p}=\left(\int_{\mathbb{B}}|f(z)|^pd\mu(z)\right)^\frac1p<+\infty.
\end{equation}
The Hardy space $H^p$ is a space of all
analytic functions $f$ in $\mathbb{B}$ with
\begin{equation}\label{eq:22}\|f\|_{H^p}=\sup_{r<1}\left(\frac1{2\pi}\int_{0}^{2\pi}|f(re^{i\theta})|^pd\theta\right)^\frac1p<+\infty.\end{equation}
It is well-known that both  Hardy space $H^p$ and Bergman space $A^p$ are a Banach space for $1\leq p$, and $H^p\subset A^p$, and both $H^p$ and $A^p$ are embeded as a closed subspaces in Lebesgue space $L^p(\mathbb{B})$, and $H^q\subset H^p$, $A^q\subset A^p$ for $q\leq p$ (for more details, see \cite{D1970,DS2004}).

Let $(X,\|\cdot\|_X)$ and $(Y,\|\cdot\|_Y)$ be two Banach space, and let $T:K\subset X\to Y$ be an operator and let $r\in\mathbb{R}$. $T$ is called \begin{itemize}
	\item[(i)]  {\em $r$-homogeneous} if $T(tx)=t^rTx$ for each  $t\in\mathbb{C}$ and all $x\in K$;
	\item[(ii)]  {\em positively homogeneous} if $T(tx)=tTx$ for each $t>0$ and all $x\in K$;
	\item[(iii)]  {\em bounded} if there is a real number $M>0$ such that  $$\|Tx\|_Y\leq M\|x\|_X, \mbox{ for all }x\in K.$$
\end{itemize}

The gamma function $\Gamma(z)$  is defined by the formula
\begin{equation}\label{eq:23}
\Gamma(z) =\int_{0}^{+\infty}e^{-t}t^{z-1}dt
\end{equation}
whenever the complex variable $z$ has a positive real part, i.e., $\Re (z)>0$. The beta function $\beta(u,v)$ is defined by the formula
\begin{equation}\label{eq:24}
\beta(u,v) =\int_{0}^{1}t^{u-1}(1-t)^{v-1}dt,\ \Re (u)>0,\ \Re (v)>0.
\end{equation}
The formula relating the beta function to the gamma function is the following:
\begin{equation}\label{eq:25}
\beta(u,v) =\frac{\Gamma(u)\Gamma(v)}{\Gamma(u+v)}.
\end{equation}
Furthermore, the gamma function has the following properties (\cite{L1972}):
\begin{itemize}
	\item[(i)] $\Gamma(1)=1$ and $\Gamma(\frac12)=\sqrt{\pi}$;
	\item[(ii)] $\Gamma(z)\Gamma(1-z)=\dfrac{\pi}{\sin(\pi z)}$ for non-integral complex numbers $z$.
	\item[(iii)] The duplication formula: $\Gamma(z)\Gamma(z+\frac12)=2^{1-2z}\sqrt{\pi}\Gamma(2z),$ i.e.,
	\begin{equation}\label{eq:26}\frac{\Gamma(z)\Gamma(z)}{\Gamma(2z)}=2^{1-2z}\frac{\Gamma(z)\Gamma(\frac12)}{\Gamma(z+\frac12)}.\end{equation}
\end{itemize}

%%%%%%%%%%%%%%%%%%%%%%%%%%%%%%%%%%%%%%%%%%%%%%%%%%%%%%%%%%%%%%%%%%%%%%%%%%%%%%%%%%%%%%%%%%%%%%%%%%%%%%%%%%%%%%
\begin{Lemma}\label{le:25} \em(\cite[Page 36, Lemma]{D1970}) If $f\in H^p$ and  $0<p<+\infty$, then
	\begin{equation}\label{eq:31}
	|f(z)|\leq \left(\frac2{1-|z|}\right)^\frac1p\|f\|_{H^p}.
	\end{equation}
\end{Lemma}

\begin{Lemma}\label{le:26} \em(\cite[Page 755, Corollary]{V1993}) If $f\in A^p$ and  $0<p<+\infty$, then
	\begin{equation}\label{eq:31}
	|f(z)|\leq \left(\frac1{1-|z|^2}\right)^\frac2p\|f\|_{A^p}.
	\end{equation}
\end{Lemma}

%%%%%%%%%%%%%%%%%%%%%%%%%%%%%%%%%%%%%%%%%%%%%%%%%%%%%%%%%%%%%%%%%%%%%%%%%%%%%%%%%%%%%%%%%%%%%%%%%%%%%%%%%%%%%%
\section{Hilbert tensor operators}

%%%%%%%%%%%%%%%%%%%%%%%%%%%%%%%%%%%%%%%%%%%%%%%%%%%%%%%%%%%%%%%%%%%%%%%%%%%%%%%%%%%%%%%%%%%%%%%%%%%%%%%%%%%%%%
\subsection{Intergral form of Hilbert tensor operator}
\begin{Lemma}\label{p31}\em Let $\mathcal{H}$ be an $m-$order infinite dimensional Hilbert tensor, and let $f(z)=\sum\limits_{k=0}^{+\infty}a_kz^k\in L^{m-1}(\mathbb{B})$. Then \begin{equation}\label{eq:31}
	\mathcal{H}(f)(z)=\sum_{k=0}^{+\infty}\left(\sum_{i_2,i_3,\cdots,i_m=0}^{+\infty}\frac{a_{i_2}a_{i_3}\cdots a_{i_m}}{k+i_2+i_3+\cdots+i_m+1}\right)z^k
	\end{equation}
	is a well-defined analytic function  on the unit disc $\mathbb{B}$. Furthermore, $\mathcal{H}(f)(z)$ is a well-defined on Hardy space H$^p$ or on Bergman space $A^p$ ($m-1<p<+\infty$).
\end{Lemma}
%%%%%%%%%%%%%%%%%%%%%%%%%%%%%%%%%%%%%%%%%%%%%%%%%%%%%%%%%%%%%%%%%%%%%%%%%%%%%%%%%%%%%%%%%%%%%%%%%%%%%%%%%%%%%%
\begin{proof}
	
	Let $f_l(z)=\sum\limits_{k=0}^{l}a_kz^k$ for all positive integer $l$.   Obviously, $\lim\limits_{l\to\infty}f_l(z)=f(z)$, and so, $\lim\limits_{l\to\infty}(f_l(z))^{m-1}=(f(z))^{m-1}$. Thus for each $z\in \mathbb{B}$, there is a  positive integer  $N$ such that $|f_l(z)|^{m-1}\leq |f(z)|^{m-1}+1$ for  all positive integer $l>N$. So for all positive integer $l>N$, we have
$$\begin{aligned}\left|\sum_{i_2,i_3,\cdots,i_m=0}^{l}\frac{a_{i_2}a_{i_3}\cdots a_{i_m}}{i_1+i_2+\cdots+i_m+1}\right|=&\left|\sum_{i_2,i_3,\cdots,i_m=0}^l a_{i_2}a_{i_3}\cdots a_{i_m}\int^1_0s^{i_1+\cdots+i_m}ds\right|\\
=&\left|\int^1_0\left(\sum_{i_2,i_3,\cdots,i_m=0}^la_{i_2}a_{i_3}\cdots a_{i_m}s^{i_2+\cdots+i_m}\right)s^{i_1} ds\right|\\
=&\left|\int^1_0\left(\sum_{i=0}^{l}a_is^i\right)^{m-1}s^{i_1} ds\right|\\
=&\left|\int^1_0\left(f_l(s)\right)^{m-1}s^{i_1} ds\right|\\
\leq&\int^1_0|f_l(s)|^{m-1}|s|^{i_1} ds\leq\int^1_0|f_l(s)|^{m-1} ds\\\leq& \int^1_0(|f(s)|^{m-1}+1) ds<+\infty\ \ (\mbox{since }f\in L^{m-1}(\mathbb{B})).
\end{aligned}$$	
Then, $$\left|\sum_{i_2,i_3,\cdots,i_m=0}^{+\infty}\frac{a_{i_2}a_{i_3}\cdots a_{i_m}}{i_1+i_2+\cdots+i_m+1}\right|<+\infty,$$
and hence the convergence radius of the following power series, denoted by $\mathcal{H}(f)(z)$, $$\mathcal{H}(f)(z)=\sum_{k=0}^{+\infty}\left(\sum_{i_2,i_3,\cdots,i_m=0}^{+\infty}\frac{a_{i_2}a_{i_3}\cdots a_{i_m}}{k+i_2+i_3+\cdots+i_m+1}\right)z^k$$
is greater than or equal to $1$. So $\mathcal{H}(f)(z)$ is a well-defined analytic function  on the unit disc $\mathbb{B}$.  The desired conclusions follow.
\end{proof}
%%%%%%%%%%%%%%%%%%%%%%%%%%%%%%%%%%%%%%%%%%%%%%%%%%%%%%%%%%%%%%%%%%%%%%%%%%%%%%%%%%%%%%%%%%%%%%%%%%%%%%%%%%%%%%
\begin{Lemma}\label{P32}
Let \begin{equation}\label{eq:32}\mathcal{G}(f)(z)=\int^1_0\frac{\left(f(s)\right)^{m-1}}{1-zs} ds\ (m\geq2)\end{equation} for $z\in \mathbb{B}$.
The operator $\mathcal{G}(f)(z)$ is well-defined on Hardy space H$^p$ ($m-1<p<+\infty$) or on Bergman space $A^p$ ($2(m-1)<p<+\infty$).% or on Korenblum space $A^{-\tau}$ ($0<\tau<\frac1{m-1}$).
\end{Lemma}
%%%%%%%%%%%%%%%%%%%%%%%%%%%%%%%%%%%%%%%%%%%%%%%%%%%%%%%%%%%%%%%%%%%%%%%%%%%%%%%%%%%%%%%%%%%%%%%%%%%%%%%%%%%%%%	
\begin{proof}
(1) For $f\in H^p$, from Lemma \ref{le:25} and the fact that
$$\frac1{|1-zs|}\leq\frac1{1-|z||s|}\leq\frac1{1-|z|},$$
it follows that $$\begin{aligned}|\mathcal{G}(f)(z)|\leq&\int^1_0\frac{\left(|f(s)|\right)^{m-1}}{|1-zs|}ds\\
\leq&\int^1_0\frac{\left(\left(\frac2{1-s}\right)^\frac1p\|f\|_{H^p}\right)^{m-1}}{1-|z|} ds\\
=&\frac{2^{\frac{m-1}{p}}\|f\|_{H^p}^{m-1}}{1-|z|}\int^1_0\frac1{(1-s)^{\frac{m-1}{p}}}ds
<+\infty\ (\frac{m-1}{p}<1)\end{aligned}$$
since the integral $\int^1_0\frac1{(1-s)^r}ds$ converges for $r<1$.

(2) For $f\in A^p$,  it follows  from Lemma \ref{le:26} that
 $$\begin{aligned}|\mathcal{G}(f)(z)|
\leq&\int^1_0\frac{\left(\left(\frac1{1-s}\right)^\frac2p\|f\|_{A^p}\right)^{m-1}}{1-|z|} ds\\
=&\frac{\|f\|_{A^p}^{m-1}}{1-|z|}\int^1_0\frac1{(1-s)^{\frac{2(m-1)}{p}}}ds
<+\infty\ (\frac{2(m-1)}{p}<1).\end{aligned}$$

%(3) For $f\in A^{-\tau}$,  it follows  from the definition of $\|f\|_{ A^{-\tau}}$ that
%$$\begin{aligned}|\mathcal{G}(f)(z)|
%\leq&\int^1_0\frac{\left(\frac1{(1-s)^\tau}\|f\|_{A^{-\tau}}\right)^{m-1}}{1-|z|} ds\\
%=&\frac{\|f\|_{A^{-\tau}}^{m-1}}{1-|z|}\int^1_0\frac1{(1-s)^{\tau(m-1)}}ds
%<+\infty\ (\tau(m-1)<1).\end{aligned}$$
The desired conclusions follow.
\end{proof}
%%%%%%%%%%%%%%%%%%%%%%%%%%%%%%%%%%%%%%%%%%%%%%%%%%%%%%%%%%%%%%%%%%%%%%%%%%%%%%%%%%%%%%%%%%%%%%%%%%%%%%%%%%%%%%
\begin{Lemma}\label{P33}\em Let $\mathcal{H}$ be an m-order infinite dimensional Hilbert tensor, and let $f\in H^p$ ($m-1< p<+\infty$) or  $f\in A^p$ ($2(m-1)<p<+\infty$).% or $f\in A^{-\tau}$ ($0<	\tau<\frac1{m-1}$).
	 Then  for each $z\in \mathbb{B}$,
	\begin{itemize}
		\item[(i)] $\mathcal{H}(f)(z)=\mathcal{G}(f)(z)=\int^1_0\frac{\left(f(s)\right)^{m-1}}{1-zs} ds$;
		\item[(ii)]$\mathcal{H}(f)(z)=\mathcal{G}(f)(z)=\int^1_0\left(f(\frac{s}{(s-1)z+1})\right)^{m-1}\frac1{(s-1)z+1}ds$.
	\end{itemize}
\end{Lemma}
%%%%%%%%%%%%%%%%%%%%%%%%%%%%%%%%%%%%%%%%%%%%%%%%%%%%%%%%%%%%%%%%%%%%%%%%%%%%%%%%%%%%%%%%%%%%%%%%%%%%%%%%%%%%%%	
	\begin{proof}
(i) Let $f_l(z)=\sum\limits_{k=0}^{l}a_kz^k.$ Obviously, $\lim\limits_{l\to\infty}f_l(z)=f(z),$ and hence, $\lim\limits_{l\to\infty}\left(f_l(z)\right)^{m-1}=(f(z))^{m-1}.$ Now we may define a polynomial
	 $$\mathcal{H}(f_l)(z)=\sum_{k=0}^{+\infty}\left(\sum_{i_2,i_3,\cdots,i_m=0}^{l}\frac{a_{i_2}a_{i_3}\cdots a_{i_m}}{k+i_2+i_3+\cdots+i_m+1}\right)z^k.
$$
Then we have
	$$\begin{aligned}
	\mathcal{H}(f_l)(z)=&\sum_{k=0}^{+\infty}z^k\sum_{i_2,i_3,\cdots,i_m=0}^la_{i_2}a_{i_3}\cdots a_{i_m}\int^1_0s^{k+i_2+\cdots+i_m}ds\\
=&\sum_{k=0}^{+\infty}z^k\int^1_0\left(\sum_{i_2,i_3,\cdots,i_m=0}^la_{i_2}a_{i_3}\cdots a_{i_m}s^{i_2+\cdots+i_m}\right)s^k ds\\
	=&\sum_{k=0}^{+\infty}z^k\int^1_0\left(\sum_{i=0}^{l}a_is^i\right)^{m-1}s^{k} ds\\
	=&\sum_{k=0}^{+\infty}\int^1_0\left(f_l(s)\right)^{m-1}(zs)^{k} ds\\
		=&\int^1_0\left(f_l(s)\right)^{m-1}\left(\sum_{k=0}^{+\infty}(zs)^{k}\right) ds\\
	=&\int^1_0\frac{\left(f_l(s)\right)^{m-1}}{1-zs} ds.
	\end{aligned}$$
	
For $z\in \mathbb{B}$ and $p>m-1$, it is obvious that the fact that $f(z)\in H^p$ implies that $(f(z))^{m-1}\in H^\frac{p}{m-1}$, and hence, $\left(f(z)\right)^{m-1}-\left(f_l(z)\right)^{m-1}\in H^\frac{p}{m-1}$. Furthermore, from Lemma \ref{le:25}, it follows that
$$\begin{aligned}
	|\mathcal{H}(f_l)(z)-\mathcal{G}(f)(z)|=&\left|\int^r_0\frac{\left(f_l(s)\right)^{m-1}-\left(f(s)\right)^{m-1}}{1-zs} ds\right|\\
	\leq&\int^1_0\frac{|\left(f_l(s)\right)^{m-1}-\left(f(s)\right)^{m-1}|}{|1-zs|}ds\\
	\leq&\int^1_0\frac{\left(\frac2{1-s}\right)^\frac{m-1}p\|f_l^{m-1}-f^{m-1}\|_{H^\frac{p}{m-1}}}{1-|z|} ds\\
=&\left(\frac{2^{\frac{m-1}{p}}}{1-|z|}\int^1_0\frac1{(1-s)^{\frac{m-1}{p}}}ds\right)\|f_l^{m-1}-f^{m-1}\|_{H^\frac{p}{m-1}}.
	\end{aligned}$$
Therefore, for each $z\in \mathbb{B}$, $$\lim_{l\to\infty}\mathcal{H}(f_l)(z)=\sum_{k=0}^{+\infty}\left(\sum_{i_2,i_3,\cdots,i_m=0}^{+\infty}\frac{a_{i_2}a_{i_3}\cdots a_{i_m}}{k+i_2+i_3+\cdots+i_m+1}\right)z^k=\mathcal{G}(f)(z).$$
 Then $\mathcal{G}(f)(z)$  defines an analytic function $\mathcal{H}(f)(z)=\lim\limits_{l\to\infty}\mathcal{H}(f_l)(z)$. That is,
$$\mathcal{H}(f)(z)=\mathcal{G}(f)(z)=\int^1_0\frac{\left(f(s)\right)^{m-1}}{1-zs} ds$$
for each $f\in H^p$ ($m-1<p<+\infty$).

Similarly, for $f\in A^p$, it follows from Lemma \ref{le:26} that $$\begin{aligned}
|\mathcal{H}(f_l)(z)-\mathcal{G}(f)(z)|
\leq&\int^1_0\frac{|\left(f_l(s)\right)^{m-1}-\left(f(s)\right)^{m-1}|}{|1-zs|}ds\\
\leq&\int^1_0\frac{\left(\frac1{1-s}\right)^\frac{2(m-1)}p\|f_l^{m-1}-f^{m-1}\|_{A^\frac{p}{m-1}}}{1-|z|} ds\\
=&\left(\frac{1}{1-|z|}\int^1_0\frac1{(1-s)^{\frac{2(m-1)}{p}}}ds\right)\|f_l^{m-1}-f^{m-1}\|_{A^\frac{p}{m-1}},
\end{aligned}$$
and so, $\mathcal{H}(f)(z)=\mathcal{G}(f)(z)$ for every $f\in A^p$ ($2(m-1)<p<+\infty$).

%For $f\in A^{-\tau}$, it follows from  the definition of $\|f\|_{A^{-\tau}}$ that   $(f(z))^{m-1}\in   A^{-\tau(m-1)}$, and hence, $\left(f(z)\right)^{m-1}-\left(f_l(z)\right)^{m-1}\in A^{-\tau(m-1)}$. Therefore, we have $$\begin{aligned}
%|\mathcal{H}(f_l)(z)-\mathcal{G}(f)(z)|\leq&\int^1_0\frac{|\left(f_l(s)\right)^{m-1}-\left(f(s)\right)^{m-1}|}{|1-zs|}ds\\\leq&\int^1_0\frac{\frac1{(1-s)^{\tau(m-1)}}\|f_l^{m-1}-f^{m-1}\|_{A^{-\tau(m-1)}}}{1-|z|} ds\\=&\left(\frac{1}{1-|z|}\int^1_0\frac1{(1-s)^{\tau(m-1)}}ds\right)\|f_l^{m-1}-f^{m-1}\|_{A^{-\tau(m-1)}},\end{aligned}$$
%and so, $\mathcal{H}(f)(z)=\mathcal{G}(f)(z)$ for every $f\in A^{-\tau}$ ($0<\tau<\frac1{m-1}$).

(ii) Given $f\in H^p$ ($m-1<p<+\infty$) or $f\in A^p$ ($2(m-1)<p<+\infty$), %or  $f\in A^{-\tau}$ ($0<\tau<\frac1{m-1}$),
the integral $\mathcal{G}(f)(z)$ is independent of the path of integration. Then for $z\in \mathbb{B}$, we may choose the path of integration $$s(t)=\frac{t}{(t-1)z+1}, 0\leq t\leq1,$$ and hence
	$$s'(t)=\frac{ds(t)}{dt}=\frac{((t-1)z+1)-tz}{((t-1)z+1)^2}=\frac{1-z}{((t-1)z+1)^2}.$$
So we have
$$\begin{aligned}
\mathcal{G}(f)(z)=&\int^1_0\frac{\left(f(\frac{t}{(t-1)z+1})\right)^{m-1}}{1-z\frac{t}{(t-1)z+1}}s'(t) dt\\
=&\int^1_0\left(f(\frac{t}{(t-1)z+1})\right)^{m-1}\frac{(t-1)z+1}{(t-1)z+1-zt}\frac{1-z}{((t-1)z+1)^2} dt\\
=&\int^1_0\left(f(\frac{t}{(t-1)z+1})\right)^{m-1}\frac1{(t-1)z+1}dt.
\end{aligned}$$ The desired conclusion follows.
\end{proof}

%%%%%%%%%%%%%%%%%%%%%%%%%%%%%%%%%%%%%%%%%%%%%%%%%%%%%%%%%%%%%%%%%%%%%%%%%%%%%%%%%%%%%%%%%%%%%%%%%%%%%%%%%%%%%%

\subsection{Boundedness of Hilbert tensor operator}
%%%%%%%%%%%%%%%%%%%%%%%%%%%%%%%%%%%%%%%%%%%%%%%%%%%%%%%%%%%%%%%%%%%%%%%%%%%%%%%%%%%%%%%%%%%%%%%%%%%%%%%%%%%%%%
\begin{Theorem}\label{th:34}
	Let $\mathcal{H}$ be an m-order infinite dimensional Hilbert tensor, and let $\mathcal{H}(f)$ be as in Lemma \ref{p31}. Then $\mathcal{H}$ is bounded and $(m-1)$-homogeneous on Bergmans space $A^{p(m-1)}$ for $2<p<+\infty$, and satisfies the following:
	\begin{itemize}
		\item[(i)] If $4\leq p<+\infty$ and $f\in A^{p(m-1)}$, then \begin{equation}\label{eq:33}\|\mathcal{H}(f)\|_{A^p}\leq \frac{\pi}{\sin(\frac{2\pi}p)}\|f\|_{A^{p(m-1)}}^{m-1}.\end{equation}
			\item[(ii)] If $2<p\leq4$ and $f\in A^{p(m-1)}$, then \begin{equation}\label{eq:34}\|\mathcal{H}(f)\|_{A^p}\leq M\|f\|_{A^{p(m-1)}}^{m-1},\end{equation}
		where $M=4^{\frac4p-1}\sqrt{\pi}\frac{\Gamma(1-\frac2p)}{\Gamma(\frac32-\frac2p)}.$
	\end{itemize}
\end{Theorem}
%%%%%%%%%%%%%%%%%%%%%%%%%%%%%%%%%%%%%%%%%%%%%%%%%%%%%%%%%%%%%%%%%%%%%%%%%%%%%%%%%%%%%%%%%%%%%%%%%%%%%%%%%%%%%%	
\begin{proof} Let $$\varphi(t,z)=\frac{t}{(t-1)z+1}\mbox{ and }\psi(t,z)=\frac{1}{(t-1)z+1}$$ for all $z\in \mathbb{B}$ and all real number $t$ with $0<t<1.$ Then
$$\frac{\partial\varphi(t,z)}{\partial z}= \frac{-t(t-1)}{((t-1)z+1)^2}
=t(1-t)(\psi(t,z))^2.$$
%So it follows that $$\left|\frac{\partial\varphi(t,z)}{\partial z}\right|=t(1-t)|\psi(t,z)|^2\leq1.$$

%%%%%%%%%%%%%%%%%%%%%%%%%%%%%%%%%%%%%%%%%%%%%%%%%%%%%%%%%%%%%%%%%%%%%%%%%%%%%%%%%%%%%%%%%%%%%%%%%%%%%%%%%%%%%%	
Let $\mathcal{T}_t(f)(z)=\psi(t,z)\left(f(\varphi(t,z))\right)^{m-1}.$ Then for each $t\in(0,1)$, we have
$$\begin{aligned}\|\mathcal{T}_t(f)\|_{A^p}^p=&\int_{\mathbb{B}}|\psi(t,z)|^p\left|\left(f(\varphi(t,z))\right)^{m-1}\right|^pd\mu(z)\\
=&\int_{\mathbb{B}}|\psi(t,z)|^{p-4}\left|f(\varphi(t,z))\right|^{p(m-1)}|\psi(t,z)|^4d\mu(z)\\
=&\int_{\mathbb{B}}|\psi(t,z)|^{p-4}\left|f(\varphi(t,z))\right|^{p(m-1)}\frac1{t^2(1-t)^2}\left|\frac{\partial\varphi(t,z)}{\partial z}\right|^2d\mu(z)\\
=&\frac1{t^2(1-t)^2}\int_{\mathbb{B}}|\psi(t,z)|^{p-4}\left|f(\varphi(t,z))\right|^{p(m-1)}\left|\frac{\partial\varphi(t,z)}{\partial z}\right|^2d\mu(z).
\end{aligned}$$

(i) For $+\infty> p\geq4$ and each $t\in(0,1)$, we have $$\psi(t,z)=\frac{\varphi(t,z)}{t},$$
$$|\varphi(t,z)|=\frac{t}{|(t-1)z+1|}\leq \frac{t}{1-|t-1||z|}\leq\frac{t}{1-(1-t)}=1$$
and furthermore,
$$\begin{aligned} \|\mathcal{T}_t(f)\|_{A^p}^p=&\frac1{t^2(1-t)^2}\int_{\mathbb{B}}\left|\frac{\varphi(t,z)}{t}\right|^{p-4}\left|f(\varphi(t,z))\right|^{p(m-1)}\left|\frac{\partial\varphi(t,z)}{\partial z}\right|^2d\mu(z)\\
=&\frac1{t^{p-2}(1-t)^2}\int_{\mathbb{B}}|\varphi(t,z)|^{p-4}\left|f(\varphi(t,z))\right|^{p(m-1)}\left|\frac{\partial\varphi(t,z)}{\partial z}\right|^2d\mu(z)\\
\leq&\frac1{t^{p-2}(1-t)^2}\int_{\mathbb{B}}\left|f(\varphi(t,z))\right|^{p(m-1)}\left|\frac{\partial\varphi(t,z)}{\partial z}\right|^2d\mu(z)\\
=&\frac1{t^{p-2}(1-t)^2}\int_{\mathbb{D}}\left|f(y)\right|^{p(m-1)}d\mu(y),
\end{aligned}$$
where $y=\varphi(t,z)$, $\mathbb{D}=\{y=\varphi(t,z); z\in \mathbb{B}\}$ and $d\mu(y)=\left|\frac{\partial\varphi(t,z)}{\partial z}\right|^2d\mu(z).$ Therefore, we have
\begin{align}\|\mathcal{T}_t(f)\|_{A^p}\leq&\frac1{t^{1-\frac2p}(1-t)^{\frac2p}}\left(\int_{\mathbb{D}}\left|f(y)\right|^{p(m-1)}d\mu(y)\right)^\frac1p\nonumber\\
=&\frac1{t^{1-\frac2p}(1-t)^{\frac2p}}\left(\left(\int_{\mathbb{D}}\left|f(y)\right|^{p(m-1)}d\mu(y)\right)^\frac1{p(m-1)}\right)^{m-1}\nonumber\\
=&t^{\frac2p-1}(1-t)^{-\frac2p}\|f\|_{A^{p(m-1)}}^{m-1}.\label{eq:35}
\end{align}

%%%%%%%%%%%%%%%%%%%%%%%%%%%%%%%%%%%%%%%%%%%%%%%%%%%%%%%%%%%%%%%%%%%%%%%%%%%%%%%%%%%%%%%%%%%%%%%%%%%%%%%%%%%%%%
From the equality  $$\mathcal{H}(f)(z)=\int^1_0\psi(t,z)\left(f(\varphi(t,z))\right)^{m-1}dt=\int^1_0\mathcal{T}_t(f)(z)dt$$
and the Minkowski's integral inequality, it follows that
\begin{align}\|\mathcal{H}(f)\|_{A^p}=&\left(\int_{\mathbb{B}}\left|\mathcal{H}(f)(z)\right|^pd\mu(z)\right)^\frac1p\nonumber\\
=&\left(\int_{\mathbb{B}}\left|\int^1_0\mathcal{T}_t(f)(z)dt\right|^pd\mu(z)\right)^\frac1p\nonumber\\
\leq&\int^1_0\left(\int_{\mathbb{B}}\left|\mathcal{T}_t(f)(z)\right|^pd\mu(z)\right)^\frac1pdt\nonumber\\
=&\int^1_0\|\mathcal{T}_t(f)\|_{A^p}dt,\label{eq:36}
\end{align}
and hence, using \eqref{eq:35}, we have
$$\begin{aligned}
\|\mathcal{H}(f)\|_{A^p}\leq&\left(\int^1_0t^{\frac2p-1}(1-t)^{(1-\frac2p)-1}dt\right)\|f\|_{A^{p(m-1)}}^{m-1}\\
=&\beta(\frac2p,1-\frac2p)\|f\|_{A^{p(m-1)}}^{m-1}\\
=&\frac{\Gamma(\frac2p)\Gamma(1-\frac2p)}{\Gamma(\frac2p+1-\frac2p)}\|f\|_{A^{p(m-1)}}^{m-1}\\
=&\frac\pi{\sin(\frac{2\pi}{p})}\|f\|_{A^{p(m-1)}}^{m-1}.
\end{aligned}$$

(ii) For $2<p\leq4$ and each $t\in(0,1)$, we also have $$|(\psi(t,z))^{-1}|=|(t-1)z+1|\leq2,$$ and so,
$$\begin{aligned}\|\mathcal{T}_t(f)\|_{A^p}^p
=&\frac1{t^2(1-t)^2}\int_{\mathbb{B}}\left|\psi(t,z)\right|^{p-4}\left|f(\varphi(t,z))\right|^{p(m-1)}\left|\frac{\partial\varphi(t,z)}{\partial z}\right|^2d\mu(z)\\
=&\frac1{t^2(1-t)^2}\int_{\mathbb{B}}\left|(\psi(t,z))^{-1}\right|^{4-p}\left|f(\varphi(t,z))\right|^{p(m-1)}\left|\frac{\partial\varphi(t,z)}{\partial z}\right|^2d\mu(z)\\
\leq&\frac{2^{4-p}}{t^2(1-t)^2}\int_{\mathbb{B}} \left|f(\varphi(t,z))\right|^{p(m-1)}\left|\frac{\partial\varphi(t,z)}{\partial z}\right|^2d\mu(z)\\
=&\frac{2^{4-p}}{t^2(1-t)^2}\int_{\mathbb{D}}\left|f(y)\right|^{p(m-1)}d\mu(y),
\end{aligned}$$
where $y=\varphi(t,z)$, $\mathbb{D}=\{y=\varphi(t,z); z\in \mathbb{B}\}$ and $d\mu(y)=\left|\frac{\partial\varphi(t,z)}{\partial z}\right|^2d\mu(z).$ Therefore, we have
\begin{align}\|\mathcal{T}_t(f)\|_{A^p}\leq&\frac{2^{\frac4p-1}}{t^{\frac2p}(1-t)^{\frac2p}}\left(\left(\int_{\mathbb{D}}\left|f(y)\right|^{p(m-1)}d\mu(y)\right)^\frac1{p(m-1)}\right)^{m-1}\nonumber\\
=&2^{\frac4p-1}t^{-\frac2p}(1-t)^{-\frac2p}\|f\|_{A^{p(m-1)}}^{m-1}.\label{eq:37}
\end{align}
From \eqref{eq:36}, \eqref{eq:37} and the duplication formula \eqref{eq:26} of the gamma function $\Gamma(\cdot)$, it follows that
$$\begin{aligned}
\|\mathcal{H}(f)\|_{A^p}\leq&\int^1_0\|\mathcal{T}_t(z)\|_{A^p}dt\\
\leq&2^{\frac4p-1}\left(\int^1_0t^{(1-\frac2p)-1}(1-t)^{(1-\frac2p)-1}dt\right)\|f\|_{A^{p(m-1)}}^{m-1}\\
=&2^{\frac4p-1}\beta(1-\frac2p,1-\frac2p)\|f\|_{A^{p(m-1)}}^{m-1}\\
=&2^{\frac4p-1}\frac{\Gamma(1-\frac2p)\Gamma(1-\frac2p)}{\Gamma(2-\frac4p)}\|f\|_{A^{p(m-1)}}^{m-1}\\
=&2^{\frac4p-1}\left(2^{1-2(1-\frac2p)}\sqrt{\pi}\frac{\Gamma(1-\frac2p)}{\Gamma(1-\frac2p+\frac12)}\right)\|f\|_{A^{p(m-1)}}^{m-1}\\
=&4^{\frac4p-1}\sqrt{\pi}\frac{\Gamma(1-\frac2p)}{\Gamma(\frac32-\frac2p)}\|f\|_{A^{p(m-1)}}^{m-1}.
\end{aligned}$$
 The desired conclusions follow.
\end{proof}
%%%%%%%%%%%%%%%%%%%%%%%%%%%%%%%%%%%%%%%%%%%%%%%%%%%%%%%%%%%%%%%%%%%%%%%%%%%%%%%%%%%%%%%%%%%%%%%%%%%%%%%%%%%%%%	

 Define an operator $T_\mathcal{H} : A^{p(m-1)} \to A^p$
by the formula
\begin{equation} \label{TH}
T_\mathcal{H}(f) (z) := \begin{cases}\|f \|_{A^{p(m-1)}}^{2-m}\mathcal{H}(f)(z),\ f \neq0\\
0,\ \ \ \ \ \ \ \ \ \ \  \ \ \ \ \ \ \ \ \ \ \ \
f =0.\end{cases}
\end{equation}
When $m$ is even, define another operator $F_\mathcal{H} :  A^{p} \to  A^p$ by the formula
\begin{equation} \label{FH}
F_\mathcal{H}(f) (z) :=\left(\mathcal{H}(f)(z)\right)^{\frac1{m-1}}.
\end{equation}
Clearly, both
$F_\mathcal{H}$ and $T_\mathcal{H}$ are bounded and positively homogeneous by Theorem \ref{th:34}. So we may define the following the operator norms (\cite{SQ2013}):
\begin{equation} \label{eq:310}
\|T_\mathcal{H}\|=\sup_{\|f\|_{A^{p(m-1)}}=1}\|T_\mathcal{H}(f)\|_{A^p}\mbox{ and } \|F_\mathcal{H}\|=\sup_{\|f\|_{A^p}=1}\|F_\mathcal{H}(f)\|_{A^p}.
\end{equation}
 The
following upper bounds and properities of the operator norm may be established.
%%%%%%%%%%%%%%%%%%%%%%%%%%%%%%%%%%%%%%%%%%%%%%%%%%%%%%%%%%%%%%%%%%%%%%%%%%%%%%%%%%%%%%%%%%%%%%%%%%%%%%%%%%%%%%
\begin{Theorem}\label{th:35}
	Let $\mathcal{H}$ be an m-order infinite dimensional Hilbert tensor, and let $\mathcal{H}(f)$ be as in Lemma \ref{p31}. Then $T_\mathcal{H}$ is a bounded and positively homogeneous operator from Bergmans space $A^{p(m-1)}$ to $A^p$ for $2<p<+\infty$, and its norm satisfies the following:
	\begin{itemize}
		\item[(i)] If $4\leq p<+\infty$, then \begin{equation}\label{eq:311}\|T_\mathcal{H}\|\leq \frac{\pi}{\sin(\frac{2\pi}p)}.\end{equation}
		\item[(ii)] If $2<p\leq4$, then \begin{equation}\label{eq:312}\|T_\mathcal{H}\|\leq 4^{\frac4p-1}\sqrt{\pi}\frac{\Gamma(1-\frac2p)}{\Gamma(\frac32-\frac2p)}.\end{equation}
	\end{itemize}
\end{Theorem}
%%%%%%%%%%%%%%%%%%%%%%%%%%%%%%%%%%%%%%%%%%%%%%%%%%%%%%%%%%%%%%%%%%%%%%%%%%%%%%%%%%%%%%%%%%%%%%%%%%%%%%%%%%%%%%
\begin{proof}
	It follows the definition \eqref{TH} of the operator $T_\mathcal{H}$  that
	$$\|T_\mathcal{H}(f)\|_{A^p}=\|\|f \|_{A^{p(m-1)}}^{2-m}\mathcal{H}(f)\|_{A^p}= \|f  \|_{A^{p(m-1)}}^{2-m}\|\mathcal{H}(f)\|_{A^p}.$$
	Then an application of Theorem \ref{th:34} yields the following:
		\begin{itemize}
		\item[(i)] For $4\leq p<+\infty$,
	$$\|T_\mathcal{H}(f)\|_{A^p}\leq \|f\|_{A^{p(m-1)}}^{2-m}\left( \frac{\pi}{\sin(\frac{2\pi}p)}\|f\|_{A^{p(m-1)}}^{m-1}\right)=\frac{\pi}{\sin(\frac{2\pi}p)}\|f\|_{A^{p(m-1)}}. $$
		\item[(ii)] For $2<p\leq4$,
	$$\begin{aligned}\|T_\mathcal{H}(f)\|_{A^p}\leq& \|f\|_{A^{p(m-1)}}^{2-m}\left( 4^{\frac4p-1}\sqrt{\pi}\frac{\Gamma(1-\frac2p)}{\Gamma(\frac32-\frac2p)}\|f\|_{A^{p(m-1)}}^{m-1}\right)\\=&4^{\frac4p-1}\sqrt{\pi}\frac{\Gamma(1-\frac2p)}{\Gamma(\frac32-\frac2p)}\|f\|_{A^{p(m-1)}}. \end{aligned}$$
\end{itemize}
So the desired conclusions directly follow from the definition of operator norm \eqref{eq:310}.	
\end{proof}
%%%%%%%%%%%%%%%%%%%%%%%%%%%%%%%%%%%%%%%%%%%%%%%%%%%%%%%%%%%%%%%%%%%%%%%%%%%%%%%%%%%%%%%%%%%%%%%%%%%%%%%%%%%%%%
\begin{Theorem}\label{th:36}
	Let $\mathcal{H}$ be an m-order infinite dimensional Hilbert tensor, and let  $\mathcal{H}(f)$ be as in Lemma \ref{p31}. Then $F_\mathcal{H}$ is a bounded and positively homogeneous operator from Bergmans space $A^{p}$ to $A^p$ for $2(m-1)<p<+\infty$, and its norm satisfies the following:
	\begin{itemize}
		\item[(i)] If $4(m-1)\leq p<+\infty$, then \begin{equation}\label{eq:311}\|F_\mathcal{H}\|\leq \left(\frac{\pi}{\sin(\frac{2(m-1)\pi}p)}\right)^\frac1{m-1}.\end{equation}
		\item[(ii)] If $2(m-1)<p\leq4(m-1)$, then \begin{equation}\label{eq:312}\|F_\mathcal{H}\|\leq4^{\frac4p}\left(\frac{\sqrt{\pi}\Gamma(1-\frac{2(m-1)}p)}{4\Gamma(\frac32-\frac{2(m-1)}p)}\right)^\frac1{m-1}.\end{equation}
	\end{itemize}
\end{Theorem}
%%%%%%%%%%%%%%%%%%%%%%%%%%%%%%%%%%%%%%%%%%%%%%%%%%%%%%%%%%%%%%%%%%%%%%%%%%%%%%%%%%%%%%%%%%%%%%%%%%%%%%%%%%%%%%
\begin{proof}
	It follows the definition \eqref{FH} of the operator $F_\mathcal{H}$
	and the Minkowski's integral inequality  that
	 \begin{align}\|F_\mathcal{H}(f)\|_{A^p}=&\left(\int_{\mathbb{B}}\left|\left(\mathcal{H}(f)(z)\right)^{\frac1{m-1}}\right|^pd\mu(z)\right)^\frac1p\nonumber\\
	=&\left(\left(\int_{\mathbb{B}}\left|\mathcal{H}(f)(z)\right|^{\frac{p}{m-1}}d\mu(z)\right)^\frac{m-1}p\right)^{\frac1{m-1}}\nonumber\\
	=&\left(\left(\int_{\mathbb{B}}\left|\int^1_0\mathcal{T}_t(f)(z)dt\right|^{\frac{p}{m-1}}d\mu(z)\right)^\frac{m-1}p\right)^\frac1{m-1}\nonumber\\
	 \leq&\left(\int^1_0\left(\int_{\mathbb{B}}\left|\mathcal{T}_t(f)(z)\right|^\frac{p}{m-1}pd\mu(z)\right)^\frac{m-1}pdt\right)^\frac1{m-1}\nonumber\\
	=&\left(\int^1_0\|\mathcal{T}_t(f)\|_{A^\frac{p}{m-1}}dt\right)^\frac1{m-1},\label{eq:315}
	\end{align}
Using the proof technique of Theorem \ref{th:34} ($p$ is replaced by $\frac{p}{m-1}$),  the followings may be proved easily:
\begin{itemize}
	\item[(i)] For $4\leq \frac{p}{m-1}<+\infty$,
	$$\|\mathcal{T}_t(f)\|_{A^\frac{p}{m-1}}\leq t^{\frac{2(m-1)}p-1}(1-t)^{-\frac{2(m-1)}p}\|f\|_{A^{p}}^{m-1},$$
and hence, 	$$\|F_\mathcal{H}(f)\|_{A^p}\leq \left( \frac{\pi}{\sin(\frac{2(m-1)\pi}p)}\|f\|_{A^{p}}^{m-1}\right)^\frac1{m-1}=\left(\frac{\pi}{\sin(\frac{2(m-1)\pi}p)}\right)^\frac1{m-1}\|f\|_{A^{p}}. $$
	\item[(ii)] For $2<\frac{p}{m-1}\leq4$, $$\|\mathcal{T}_t(f)\|_{A^\frac{p}{m-1}}\leq 2^{\frac{4(m-1)}p-1}t^{-\frac{2(m-1)}p}(1-t)^{-\frac{2(m-1)}p}\|f\|_{A^{p}}^{m-1},$$
and hence,
	$$\begin{aligned}\|F_\mathcal{H}(f)\|_{A^p}\leq& \left( 4^{\frac{4(m-1)}p-1}\sqrt{\pi}\frac{\Gamma(1-\frac{2(m-1)}p)}{\Gamma(\frac32-\frac{2(m-1)}p)}\|f\|_{A^{p}}^{m-1}\right)^\frac1{m-1}\\=&4^{\frac4p}\left(\frac{\sqrt{\pi}\Gamma(1-\frac{2(m-1)}p)}{4\Gamma(\frac32-\frac{2(m-1)}p)}\right)^\frac1{m-1}\|f\|_{A^{p}}. \end{aligned}$$
\end{itemize}
So the desired conclusions directly follow from the definition of operator norm \eqref{eq:310}.	
\end{proof}
%%%%%%%%%%%%%%%%%%%%%%%%%%%%%%%%%%%%%%%%%%%%%%%%%%%%%%%%%%%%%%%%%%%%%%%%%%%%%%%%%%%%%%%%%%%%%%%%%%%%%%%%%%%%%%
Let $p=4$ in Theorem \ref{th:35} ((i) or (ii)) and $p=4(m-1)$ in Theorem \ref{th:36}((i) or (ii)), respectively. Then the following conclusions are easily obtained.
\begin{Corollary}\label{co:36}
Let	$\mathcal{H}$ be an $m-$order infinite dimensional Hilbert tensor, and let $\mathcal{H}(f)$ be as in Lemma \ref{p31}. Then
\begin{itemize} \item[(i)] $T_\mathcal{H}:A^{4(m-1)}\to A^4$ is a bounded and positively homogeneous operator  and $$\|T_\mathcal{H}\|\leq \pi;$$
	\item[(ii)]$F_\mathcal{H}:A^{4(m-1)}\to A^{4(m-1)}$ is a bounded and positively homogeneous operator and $$\|F_\mathcal{H}\|\leq \pi^\frac1{m-1}.$$
\end{itemize}
\end{Corollary}

\begin{remark}
	\begin{itemize}
		\item[(i)]  In this paper, the boundedness of Hilbert tensor operator is obtained on $A^p$ for $p>2(m-1)$. For $0<p\leq2(m-1)$,  whether or not  Hilbert tensor operator is  bounded on Bergman space $A^p$ or Hardy space $H^p$ or other space of analytic functions.
		\item[(ii)] Are the upper bounds of norm of Hilbert tensor operator the best in this paper?
		\item[(iii)]  May the operator norms of $T_\mathcal{H}$ and $F_\mathcal{H}$ be given the exact value?
	\end{itemize}
\end{remark}

%\section*{\bf Acknowledgment} The authors would like to thank the anonymous
% referees for their valuable suggestions which helped us to improve this manuscript.

%\bibliographystyle{amsplain}

\end{document}